\documentclass[11pt]{article}
\usepackage{color}
\usepackage{amsmath}

\usepackage{extarrows}
\usepackage{geometry}
\usepackage{authblk} 
\usepackage{hyperref} 
\usepackage{color}
\usepackage{ntheorem}

\usepackage[mathscr]{euscript}
\usepackage[utf8]{inputenc}
\def\q{\hfill\rule{1ex}{1ex}}
\def\0{\emptyset}

\def\q{\hfill\rule{1ex}{1ex}}

\newtheorem{theorem}{Theorem}[section]

\newtheorem{lemma}[theorem]{Lemma}

\newtheorem{cor}[theorem]{Corollary}

\newtheorem{conjecture}[theorem]{Conjecture}
\newenvironment{reproof}{{\noindent\it Proof of Theorem \ref{thm:egde-proper graph}.}}{\hfill $\square$\par}
\newenvironment{reprooff}{{\noindent\it Proof of Theorem \ref{thm:stronger result}.}}{\hfill $\square$\par}
\newenvironment{proof}{{\noindent\it Proof.}}{\hfill $\square$\par}

\usepackage{graphicx}

\usepackage{enumerate}
\usepackage{enumitem}

\usepackage{
	amsmath,			
	amssymb,			
	enumerate,		    
	graphicx,			
	lastpage,			
	multicol,			
	multirow,			
	pifont,			    
}

\usepackage[numbers]{natbib}

\newcommand{\lf}{\left\lfloor}
\newcommand{\rf}{\right\rfloor}
\newcommand{\lc}{\left\lceil}
\newcommand{\rc}{\right\rceil}

\newcounter{cases}
\newcounter{subcases}[cases]
\newenvironment{mycase}
{
    \setcounter{cases}{0}
    \setcounter{subcases}{0}
    \newcommand{\case}
    {
        \par\indent\stepcounter{cases}\textbf{Case \thecases.}
    }
    
}
{
    \par
}
\renewcommand*\thecases{\arabic{cases}}

\begin{document}


\title{On the minimum size of an edge-pancyclic graph of a given order}
\author[a]{Xiamiao Zhao \thanks{zxm23@mails.tsinghua.edu.cn}}
\author[b]{Yuxuan Yang \thanks{Correspondence Author. yangyx@bupt.edu.cn}}

\affil[a]{ Department of Mathematical Sciences, Tsinghua University, Beijing 100084, China}
\affil[b]{School of Mathematical Sciences, Beijing University of Posts and Telecommunications, Beijing 100876, China}

\date{}

\maketitle\baselineskip 16.3pt

\begin{abstract}
A graph $G$ of order $n$ is called edge-pancyclic if, for every integer $k$ with $3 \leq k \leq n$, every edge of $G$ lies in a cycle of length $k$. Determining the minimum size $f(n)$ of a simple edge-pancyclic graph with $n$ vertices seems difficult. Recently, Li, Liu and Zhan gave both a lower bound and an upper bound on $f(n)$. In this paper, we improve their lower bound by considering a new class of graphs and improve the upper bound by constructing a family of edge-pancyclic graphs.
\end{abstract}


{\bf Keywords: Edge-pancyclic graph; minimum size; cycle; Hamilton cycle}
\vskip.3cm

\section{Introduction}
In this paper, all the graphs are simple. Let $C_k$ denote the cycle with $k$ vertices. Let $G=(V,E)$ be a graph, where $E(G)$ is the edge set and  $V(G)$ is the vertex  set.
We call a vertex with degree $d$ as $d$-degree vertex.  Let $e(G)=|E(G)|$ and $v(G)=|V(G)|$. For every vertex $v\in V(G)$, $N(v)$ is the neighborhood of $v$. Then the degree of $v$ is $d(v)=|N(v)|$ and the minimum degree of $G$ is $\delta(G)=\min\{d(v):v\in V(G)\}$.

A graph $G$ with $n$ vertices is called \textbf{pancyclic} if, for every integer $k$  with $3 \leq k \leq n$, $G$ contains a copy of  $C_k$. There are two stronger conditions which derive two interesting families of graphs.

A graph $G$ with $n$ vertices is \textbf{edge-pancyclic} (resp. vertex-pancyclic) if, for every $3\leq k\leq n$, each edge (resp. vertex) of $G$ lies in a cycle of length $k$.
Every nonempty edge-pancyclic graph is vertex-pancyclic. There is much work around these two classes of graphs \cite{broersma1997note,hendry1990extending,hobbs1976square,jesso2010hamiltonicity,randerath2002vertex,lei2022extended}.

Let $g(n)$ denote the minimum size of a vertex-pancyclic graph with $n$ vertices. Broersma \cite{broersma1997note} proved that when  $n\geq 7$, $3n/2 < g(n) \leq 5n/3$. There is no widely accepted conjecture on the precise value of $g(n)$.

Recently, Li, Liu and Zhan \cite{li2024minimum} proposed the following problem:

\textbf{Problem 1.} Determine the minimum size of an edge-pancyclic graph of order $n$.

They denoted the minimum size of an edge-pancyclic graph as $f(n)$ and gave the following result.
\begin{theorem}[\cite{li2024minimum}]\label{thm: the lower bound now}
    When $n\geq 4$,
    $$\lc\frac{3n}{2}\rc\leq f(n)\leq 2n-2.$$
\end{theorem}
The lower bound in Theorem \ref{thm: the lower bound now} can be directly obtained from the lower bound on $g(n)$, since every edge-pancyclic graph is vertex-pancyclic.
If $G$ is 3-connected, the exact size of a minimum edge-pancyclic graph was determined.
\begin{theorem}[\cite{li2024minimum}]
    Let $G$ be a 3-connected edge-pancyclic graph with $n$ vertices. Then $e(G)\geq 2n-2$ and equality holds if and only if $G=W_n$.
\end{theorem}
Here $W_n$ is the wheel graph with $n$ vertices, which is constructed by joining an extra vertex with every vertex in a copy of cycle $C_{k-1}$.

In this paper, we shall show that when $n$ is larger, the lower bound on $f(n)$ in Theorem \ref{thm: the lower bound now} can be improved to $\lc\frac{7n}{4}\rc$.
 Let $k\geq 3$. A graph $G$ is called \textbf{$k$-edge-proper} if each edge in $G$  lies in a $C_i$ for every $3\leq i\leq k$ and also lies in a Hamilton cycle. Clearly, for any $n\geq k\geq 3$, every edge-pancyclic graph of order of $n$ is $k$-edge-proper.

 We have the following results.
\begin{theorem}\label{thm:egde-proper graph}
    Let $G$ be a 3-edge-proper graph of order $n\geq 6$. Then $e(G)\geq \lc\frac{5n}{3}\rc$.
\end{theorem}
The bound is tight by the following construction.

Let $k\geq 2$. We will define a graph $A_{3k}$ with $3k$ vertices, a graph $B_{3k+1}$ with $3k+1$ vertices and two graphs $D_{3k+2}^1$ and $D_{3k+2}^2$ with $3k+2$ vertices as follows.

Let $C_{2k}=v_1v_2\dots v_{2k}v_1$ be a cycle with $2k$ vertices. The graph $A_{3k}$ is obtained from $C_{2k}$ by adding $k$ new vertices $\{u_1,u_2,\dots,u_k\}$ and adding  edges $u_iv_{2i-1}$, $u_iv_{2i}$ and $u_iv_{2i+1}$ for $i=1,\dots, k$ where $v_{2k+1}=v_1$.

Let $C_{2k+1}=v_1v_2\dots v_{2k+1}v_1$ be a cycle with $2k+1$ vertices. The graph $B_{3k+1}$ is obtained from $C_{2k+1}$ by adding $k$ new vertices $\{u_1,u_2,\dots,u_k\}$ and adding  edges $u_iv_{2i-1}$, $u_iv_{2i}$ and $u_iv_{2i+1}$ for $i=1,\dots, k$ and one edge $v_{2k+1}v_2$.

Let $C_{2k+2}=v_1v_2\dots v_{2k+2}v_1$ be a cycle with $2k+2$ vertices. The graph $D_{3k+2}^1$ is obtained from $C_{2k+1}$ by adding $k$ new vertices $\{u_1,u_2,\dots,u_k\}$ and adding edges $u_iv_{2i-1}$, $u_iv_{2i}$ and $u_iv_{2i+1}$ for $i=1,\dots, k$ and two edges $v_{2k+1}v_1$, $v_{2k+2}v_2$. The graph $D_{3k+2}^2$ is obtained from $D_{3k+2}^1$ by deleting the edge $v_{2k+1}v_1$ and adding the edge $v_{2k+2}v_{2k}$.


Figure \ref{fig:eABD} shows the graphs $A_6$, $B_7$, $D_8^1$ and $D_8^2$.
\begin{figure}
    \centering
    \includegraphics[width=0.9\linewidth]{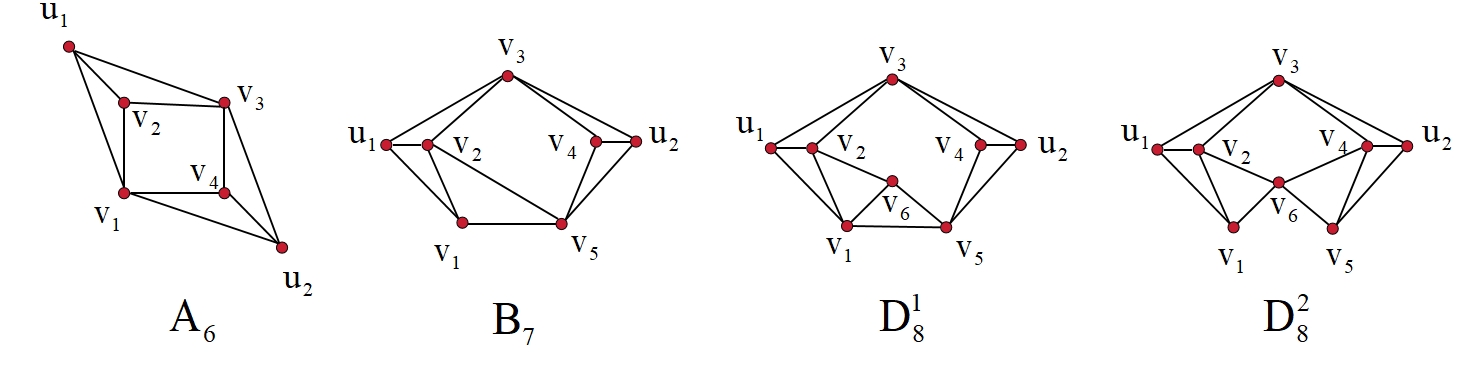}
    \caption{The extremal graphs.}
    \label{fig:eABD}
\end{figure}

When $\lc \frac{n}{3}\rc=k$, the graphs  $A_{3k},B_{3k+1},D_{3k+2}^1, D_{3k+2}^2$ are all $3$-edge-proper graphs with $\lc\frac{5n}{3}\rc$ edges. Since the graphs $D_{3k+2}^1$ and $D_{3k+2}^2$ are both extremal graphs, we declare that there exist many $3$-edge proper graphs that achieve this bound.
Actually, We can also give a lower bound of the minimum size of a $4$-edge proper graph..
\begin{theorem}\label{thm:stronger result}
    Let $G$ be a 4-edge-proper graph of order $n\geq 7$. Then $e(G)\geq \lc\frac{7n}{4}\rc$.
\end{theorem}
Notice that if a graph with $n\geq k$ vertices is edge-pancyclic, it is $k$-edge-proper. The above theorem implies the following result.
\begin{cor}
    When $n\geq 8$, we have
    $$\lc\frac{7n}{4}\rc\leq f(n)\leq 2n-2.$$
\end{cor}
We will show Theorem \ref{thm:stronger result} is asymptotically optimal.
The fan graph $F_k$ is obtained from $P_{k-1}=v_1v_2\dots v_{k-1}$ by adding a new vertex $u$ and adding the edges $uv_i$ for $i=1,2,\dots, k-1$.
When $n=4t$ for some integer $t\geq 3$, let $G$ be the graph obtained from $C_t=w_1w_2\dots w_tw_1$ by replacing every edge $w_iw_{i+1}$ with a copy of $F_5$ with $w_i=v_1$ and $w_{i+1}=v_{4}$. Because $G$ is symmetric, it is easy to check $G$ is $4$-edge-proper and $e(G)=7t=7v(G)/4.$

From the above result, we have the following natural generalization.
\begin{conjecture}
    Let $k\geq 4$ and $n$  large enough. If a graph $G$ of order $n$  is $k$-edge-proper, then $e(G)\geq \lc (4k-9)n/(2k-4)\rc$.
\end{conjecture}
A similar construction shows this conjecture is asymptotically optimal. When $n=t(2k-4)$ for some integer $t\geq 3$, let $G$ be obtained from $C_t=w_1w_2\dots w_tw_1$ by replacing every edge $w_iw_{i+1}$ with a copy of $F_{2k-3}$ with $w_i=v_1$ and $w_{i+1}=v_{2k-4}$. Because $G$ is symmetric, it is easy to check $G$ is $k$-edge-proper and $e(G)=(4k-9)t=(4k-9)n/(2k-4).$
The conjecture above implies the following conjecture.
\begin{conjecture}
    If $G$ is edge-pancyclic with $n$ vertices and $n$ is large enough, then
    $$e(G)\geq 2n-o(n).$$
\end{conjecture}
In \cite{li2024minimum},  Li, Liu and Zhan gave an upper bound on the minimum size of an edge-pancyclic graph with $n$ vertices.
\begin{theorem}[\cite{li2024minimum}]\label{1}
    Given any integer $k\geq 3$, let $n=6k^2-5k$. Then there exists an edge-pancyclic graph with $n$ vertices and size $2n-k$.
\end{theorem}
By Theorem \ref{1}, an upper bound on $f(n)$  is $2n-\Theta(n^{\frac{1}{2}})$.
In this paper, we improve this upper bound to $2n-\Theta(\frac{n}{\ln n})$.
\begin{theorem}\label{thm: upper bound} Let $s\ge e^8$ be an integer and $\ell=\lc\frac{s}{\ln s}\rc$.
    When $n=(100s-1)s^\ell$,  there exists an edge-pancyclic graph $G$ with $n$ vertices and $e(G)\leq 2n-\frac{n}{200\ln n}.$
\end{theorem}
This paper is arranged as follows. In Section \ref{sec: section2}, we will prove Theorems \ref{thm:egde-proper graph} and  \ref{thm:stronger result}. In Section \ref{sec:section 3}, we will give a construction of an edge-pancyclic graph to prove Theorem \ref{thm: upper bound}.
\section{A lower bound on the size of a $k$-edge-proper graph}\label{sec: section2}
Let $G$ be a $k$-edge-proper graph of order $n$. First, we have a simple property of a $k$-edge-proper graph.

 \begin{lemma}\label{pro}
 Let $k\ge 3$ and $G$ be a $k$-edge-proper graph of order $n$. Then  $\delta(G)\ge 3$.
 \end{lemma}

 \begin{proof} Since $G$ contains a Hamilton cycle, $\delta(G)\ge 2$.
 Suppose there exists a vertex $v\in V(G)$ with $d(v)=2$. Let $N(v)=\{x,y\}$. Since  $vx$ lies in a triangle, $xy\in E(G)$. Then $\{x,y\}$ is a vertex cut which implies  $xy$ does not lie in a Hamilton cycle, a contradiction.
 \end{proof}

 In the following discussion, by Lemma \ref{pro}, we let  $V_3=\{u\in V(G)|~d(u)=3\}$ and $V_{4^+}=V(G)\setminus V_3$.
Now we start our proof with the easier case, Theorem \ref{thm:egde-proper graph}.
\vskip.2cm
\begin{reproof} Let $G$ be a 3-edge-proper graph of order $n\ge 6$.
    We will prove by induction on  $n$. It is  trivial when $n\le 8$. Assume $n\geq 9$.
    By Lemma \ref{pro}, we can assume $\delta(G)= 3$.
    Let $v\in V_3$ and $N(v)=\{x,y,z\}$. Since $G$ is 3-edge-proper, without loss of generality, we can assume that $xy, xz\in E(G)$.

    \noindent{\bf Claim 1.}
        We may assume $x\in V_{ 4^+}$.

    \noindent{\bf Proof of Claim 1.}
        Suppose $x\in V_3$. Then we have $yz\not\in E(G)$; otherwise $\{y,z\}$ is a vertex cut which implies that $yz$ does not lie in a Hamilton cycle. Let $H$ be the graph obtained from $G$ by contracting the subgraph $G[\{v,x,y,z\}]$. Then it is easy to check that every edge of $H$ lies in a triangle. In  $G$, every edge that is not in $\{vx,vy,vz,xy,xz\}$ belongs to a Hamilton cycle that continuously passes through the edges $yx,xv,vz$ or $yv,vx,xz$. In both cases, the Hamilton cycle of $G$ will remain a Hamilton cycle of $H$ after contracting  $G[\{v,x,y,z\}]$. Thus $H$ is $3$-edge-proper. By induction hypothesis, we have $e(H)\geq \lc5(n-3)/3\rc$. Then $e(G)\geq e(H)+ 5\geq \lc5n/3\rc.$\q

     \noindent{\bf Claim 2.}
        We may assume $y,z\in V_{ 4^+}$.

     \noindent{\bf Proof of Claim 2.}
        Assume without loss of generality that $y\in V_3$. Then we have $yz\not\in E(G)$; otherwise $\{x,z\}$ is a vertex cut which implies that $xz$ is not in a Hamilton cycle. Then there is a $y_1\in V(G)\setminus\{v,x,y,z\}$ such that  $yy_1\in E(G)$. Since $y_1y$ lies in a triangle and $y,v\in V_3$, we have $y_1x\in E(G)$. Let $H$ be a graph obtained from $G$ by contracting the edge $yv$ as a vertex $v'$. It is easy to check every edge of $H$ lies in a triangle. Then we will prove that every edge of $H$ lies in a Hamilton cycle.
        \begin{figure}
        \centering
        \includegraphics[width=0.4\linewidth]{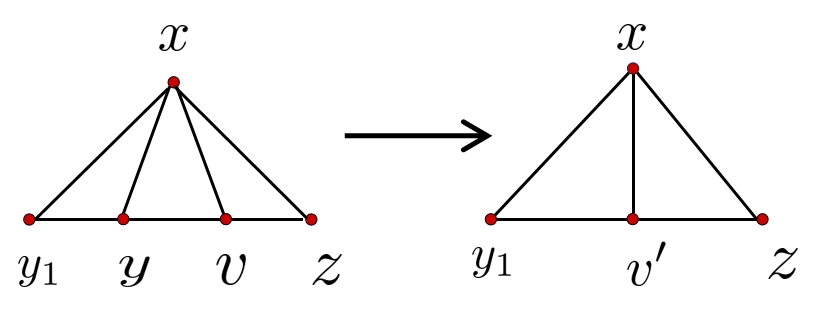}
        \caption{Contract the edge $yv$ as $v'$.}
        \label{fig:enter-label}
    \end{figure}
Let $C$ be a Hamilton cycle in $G$. If $yv\in E(C)$,  by replacing  $yv$ with $v'$, we have a Hamilton cycle in $H$. If $yv\notin E(C)$, then $C$ passes $y_1,y,x,v,z$ consecutively by $d(y)=d(v)=3$.
 Then we replace the segment $y_1yxvz$ in $C$ with $y_1v'xz$ and derive  a Hamilton cycle in $H$.
        Thus $H$ is $3$-edge-proper. By induction hypothesis, $e(G)\geq e(H)+2\geq \lc 5(n-1)/3\rc+2\geq \lc 5n/3\rc.$\q

      By Claims 1 and 2, we may assume $V_3$ is an independent set. Next, we will calculate the average degree of $G$. We define two functions $f_0,f:V(G)\to \mathbb{R}$. For every $u\in V(G)$, let $f_0(u)=d(u)$. For every vertex $v\in V_3$, we let $f(v)=24/7$, and for every vertex $u\in V_{ 4^+}$, we let $f(u)=6/7d(u)\geq 24/7.$ Let every vertex $u\in V_{4^+}$ give out $1/7$ value of $f_0$ to all of its $3$-degree neighborhood, thus its resulting value is at least $f_0(u)-1/7d(u)=d(u)-1/7d(u)=f(u)$. For any $u\in V_3$, $u$ will receive at least $1/7$ value from every one of its neighbourhoods, therefore its resulting value is $f_0(u)+3/7=3+3/7=24/7=f(u)$. Thus
     $$24/7n\leq  \sum_{w\in V(G)}f(w)\leq \sum_{w\in V(G)}f_0(w)=\sum_{w\in V(G)}d(w).$$
     So we have $e(G)\geq \lc 12/7n\rc > \lc 5n/3\rc$ when $n\geq 6$,
     and we finished the proof.
\end{reproof}

\vskip.3cm
We can have a slightly stronger result stated in Theorem \ref{thm:stronger result} with a more precise analysis.
%


\vskip.2cm

\begin{reprooff}
    Let $G$ be a 4-dege-proper graph with $n\ge 7$ vertices. We will prove by induction on $n$. When $n\le 8$, it can be trivially checked. Then we assume $n\geq 9$. By Lemma \ref{pro}, we can assume $\delta(G)= 3$.

   Let $v\in V_3$  and $N(v)=\{x,y,z\}$. Without loss of generality, we assume $xy,xz\in E(G)$. Then we have the following claims.

 \noindent{\bf Claim 3.}  $x\in V_{ 4^+}$.

 \noindent{\bf Proof of Claim 3.}
   Suppose $x\in V_3$. Then $yz\not\in E(G)$; otherwise $yz$ is not contained in any Hamilton cycle by $\{y,z\}$ being a vertex cut. Then the edge $vx$ is not in any $C_4$, a contradiction.\q

    \noindent{\bf Claim 4.}\label{claim: max at least 5}
        If $d(y)=4$, $d(z)= 4$ and $yz\not\in E(G)$, then we may assume  $d(x)\geq 5$

    \noindent{\bf Proof of Claim 4.}
        We prove this by contradiction. Suppose $d(x)=4$ by Claim 3. Since $vx$ lies in a $C_4$,  either $y,x$ or $x,z$ has a common neighbor. By symmetry, we assume $y,x$ has a common neighbor $y_1$. We will finish the proof by considering the common neighbors of $y$ and $z$.
        \begin{figure}[th]
        \centering
        \includegraphics[width=1\linewidth]{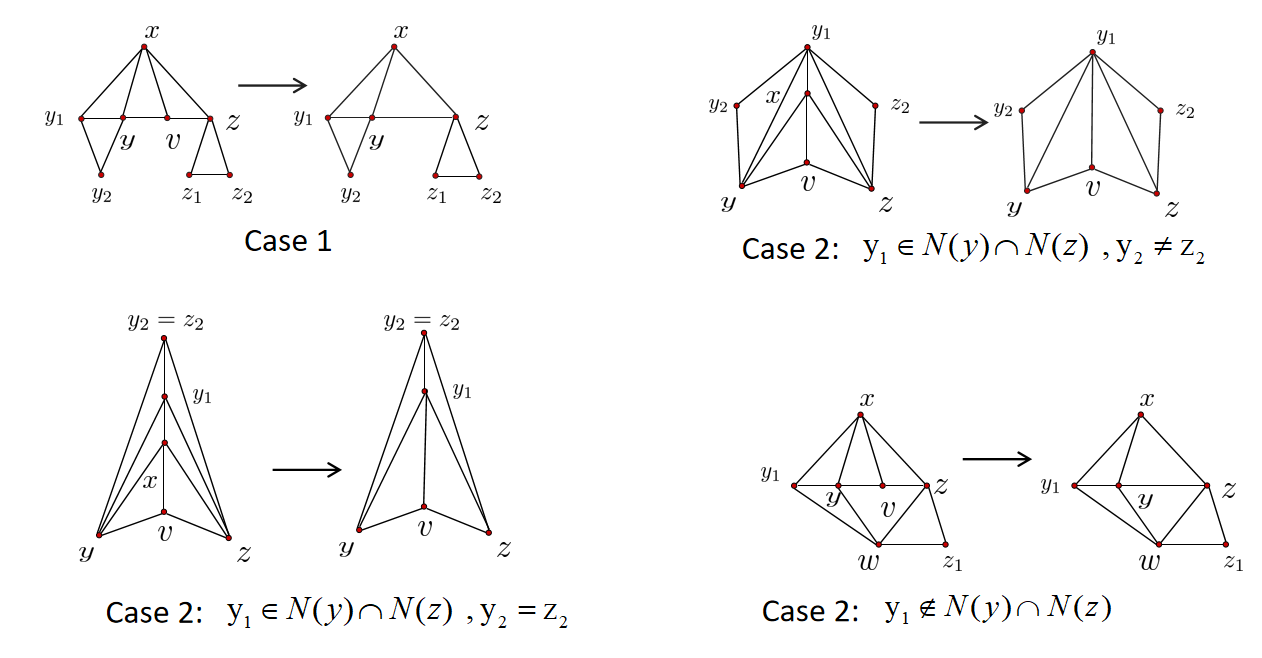}
        \caption{Two cases about  $N(y)\cap N(z)$.}
        \label{fig:enter-label}
    \end{figure}

        {\bf Case 1} $N(y)\cap N(z)=\{v,x\}$.

Let $N(y)\setminus\{v,x,y_1\}=\{y_2\}$ and $N(z)\setminus\{v,x\}=\{z_1,z_2\}$. Since $yy_2$ lies in a $C_3$ and $y_2x, y_2v\not\in E(G)$, we have $y_1y_2\in E(G)$. Since $z_1z$ lies in a $C_3$ and $z_1v,z_1x\not\in E(G)$, $z_1z_2\in E(G)$ (see Figure 3).

        Let $H$ be the graph obtained from $G$ by contracting $vy$ as $y$ (see Figure 3). Then we can check that $H$ is also $4$-edge-proper.
        Because for every edge contained in the same cycle of length $3$ or $4$ as $vy$, we can find another cycle that contains the edge in $H$. Thus, every edge is contained in cycles of length $3$ and $4$. And for every Hamilton cycle that contains $v$, if it contains $xv,vy$, then we can replace the two edges with $xy$ since $xy$ is not in the Hamilton cycle. Similarly, we can replace $xv,vz$ with $xz$ and replace $yv,vz$ with $yz$ in other Hamilton cycles. Thus, every edge is still contained in a Hamilton cycle, which implies $H$ is also a $4$-edge proper graph.
        Note that a routine verification above shows that $H$ is 4-edge-proper, which will be repeatedly applied. We therefore omit the details in subsequent occurrences.
        By induction hypothesis, we have $e(G)=e(H)+2\geq \lc 7n/4\rc.$

        {\bf Case 2} $|N(y)\cap N(z)|\ge 3$.

        We first consider the case $y_1\in N(y)\cap N(z)$.
        Since $d(y)=4$, there exists a vertex $y_2\in N(y)\setminus\{y_1,x,v\}$. Since $yy_2$ is contained in a $C_3$, and $y_2x,y_2v,yz\not\in E(G)$, we have $y_2y_1\in E(G)$. Similarly, there exists a vertex $z_2\in N(z)\cap N(y_1)$ (possible $y_2=z_2$) (see Figure 3). Let $H$ be the graph obtained by contracting the edge $y_1x$. With similar proof as above, $H$ is also $4$-edge-proper. By induction hypothesis, we have $e(G)\geq \lc 7n/4\rc$.

       Now we consider the case $y_1\not\in N(y)\cap N(z)$. Then there exists a vertex $w\in (N(y)\cap N(z))\setminus \{v,z,y_1\}$ which implies $N(y)=\{x,v,y_1,w\}$. Since $yw$ lies in a $C_3$, and $wx,wv,yz\not\in E(G)$, we have $y_1w\in E(G)$. Since $wz$ lies in $C_3$, there is $z_1\in (N(w)\cap N(z))\setminus\{v,x,y_1\}$ (see Figure 3).
        Let $H$ be obtained from $G$ by contacting $yv$ as $y$. Then with similar proof as above, $H$ is also $4$-edge-proper.
        By induction hypothesis, we have that $e(G)\geq \lc 7n/4\rc.$\q

    \vskip.2cm
    \noindent{\bf  Claim 5.} $|\{y,z\}\cap V_3|\le 1$.

     \noindent{\bf Proof of Claim 5.} Suppose $y,z\in V_3$. Then $yz\not\in E(G)$; otherwise $\{x,z\}$ is a vertex cut and the edge $xz$ will not be contained in any Hamilton cycle, a contradiction.
     Let $N(y)\setminus\{v,x\}=\{y_1\}$. Since $yy_1$ lies in a $C_3$ and $yz, y_1v\not\in E(G)$, $xy_1\in E(G)$. If $y_1z\in E(G)$, then $\{x,y_1\}$ is a vertex cut and the edge $xy_1$ will not be contained in any Hamilton cycle, a contradiction. Then there is $z_1\in N(z)\setminus\{v,x,y_1\}$. We can similarly check the graph $H$ obtained from $G$ by contracting $yv$ as $v'$ is a $4$-edge-proper graph of order $n-1$. By induction hypothesis, $e(G)\geq \lc 7(n-1)/4\rc+2\geq \lc7n/4\rc.$\q


    \noindent{\bf  Claim 6.}\label{claim: when d(y)=3}
        If $y\in V_3$ (resp. $z\in V_3$), then $z\in V_{ 4^+}$ (resp. $y\in V_{ 4^+}$) and we may assume
  $d(x)\geq 5$.

   \noindent{\bf Proof of Claim 6.} Suppose $d(y)=3$. By Claim 5, $z\in V_{ 4^+}$.
     Suppose  $d(x)=4$. By the same argument as the proof of Claim 5, we have $yz\not\in E(G)$
      and $y_1x\in E(G)$, where $N(y)\setminus\{v,x\}=\{y_1\}$. If $y_1z\in E(G)$, then $\{z,y_1\}$ is a vertex cut and the edge $zy_1$ will not be contained in any Hamilton cycle, a contradiction.   So $y_1\not\in N(z)$.
           By Claim 5,
      $d(y_1)\geq 4$. We will finish the proof by considering whether $y_1$ and $z$ have common neighbors.

         {\bf Case 1} $N(y_1)\cap N(z)\not=\emptyset$.

       Let $w\in N(y_1)\cap N(z)$. Since $y_1w$ and $wz$ are all contained in some $C_3$, there exists a vertex $u_1\in (N(y_1)\cap N(w))\setminus\{v,x,y,z\}$ and $u_2\in (N(z)\cap N(w))\setminus\{v,x,y,y_1\}$ (possible $u_1=u_2$).
        Let $H$ be the graph obtained from $G$ by deleting vertices $x,y,v$ and adding an edge $y_1z$. It is easy to check that $H$ is a $4$-edge-proper graph of order  $n-3$, and $e(H)=e(G)-6$. By induction hypothesis, $$e(G)=e(H)+6\geq \lc 7(n-3)/4\rc+6\geq \lc 7n/4\rc.$$

       {\bf Case 2} $N(y_1)\cap N(z)=\emptyset$.

         Let $H$ be obtained from $G$ by contracting the subgraph $G[\{y_1,x,y,v,z\}]$. Then it is easy to check that $H$ is $4$-edge-proper graph of order $n-4$ and  $e(H)=e(G)-7$. By induction hypothesis, we have
        $$e(G)\geq e(H)+7\geq \lc7(n-4)/4\rc+7\geq \lc 7n/4\rc.$$
    \q

By Claims 3, 4 and 6, we
    can classify all  vertices in $V_3$ into three classes. Let (see Figure 4)
   \begin{align}
    A_1=&\{v\in V_3|N(v)=\{x,y,z\},xy,xz,yz\in E(G),x,y,z\in V_{4+}\},\notag\\
    A_2=&\{v\in V_3|N(v)=\{x,y,z\},xy,xz\in E(G), yz\not\in E(G),\max\{d(x),d(y),d(z)\}\ge5,\notag
      x,y,z\in V_{4+}\},\notag\\
    A_3=&\{v\in V_3|N(v)=\{x,y,z\},xy,xz\in E(G),d(x)\geq 5,y\in V_3, z\in V_{4+} \mbox{~or~}z\in V_3, y\in V_{4+} \}.\notag
    \end{align}




    \begin{figure}
        \centering
        \includegraphics[width=1\linewidth]{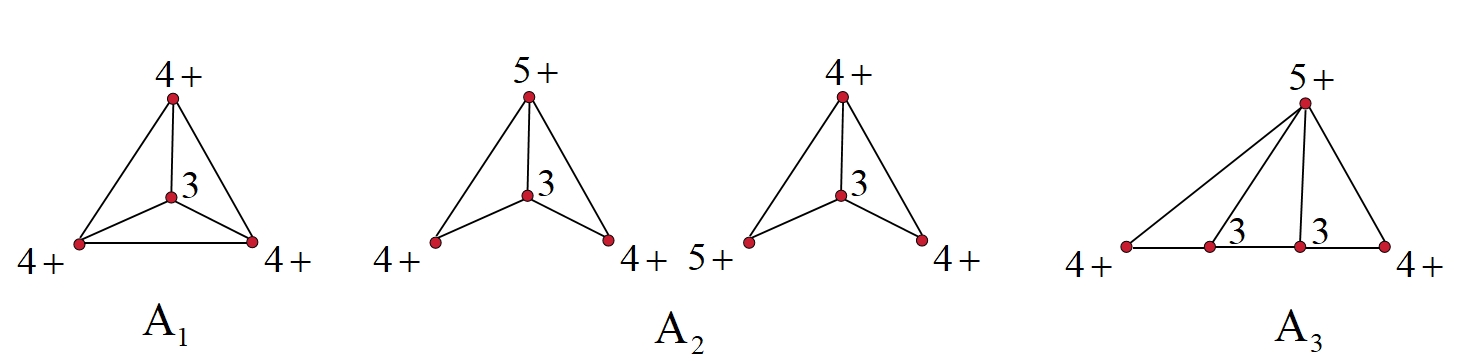}
        \caption{Three classes of $3$-degree vertices.}
        \label{fig:enter-label}
    \end{figure}

    Now we define a function $f_0: V(G)\to \mathbb{R}$ such that $f_0(u)=d(u)$ for all $u\in V(G)$.

    For every  $u\in V(G)$ with $d(u)=4$, let $u$ give out  $1/5$ to all the vertices in $N(u)\cap A_1$, $3/23$ to all the vertices in $N(u)\cap A_2$ and $2/15$ to all the vertices in $N(u)\cap A_3$.

    For every  $w\in V(G)$ with $d(w)\ge 5$, let $w$ give out $1/5$ to all the vertices in $N(w)\cap A_1$, $8/23$ to all the vertices in $N(w)\cap A_2$ and $7/15$ to all the vertices in $N(w)\cap A_3$.

    The resulting value function is denoted as $f_1: V(G)\to \mathbb{R}$.

    \noindent{\bf Claim 7.}
        For every $v\in V(G)$, $f_1(v)\geq 83/23$.

    \noindent{\bf Proof of Claim 7.}
        If $v\in A_1$, then  $f_1(v)\geq 3+3\cdot 1/5=18/5>82/23$.
    If $v\in A_2$, then  $f_1(v)\geq 3+2\cdot 3/23+8/23=83/23$.
    If $v\in A_3$, then  $f_1(v)=3+2/15+7/15=18/5>83/23$.

    Let $v\in V_{ 4^+}$ and $d(v)=4$. If  $N(v)\cap A_1=\emptyset$, then $|N(v)\cap V_3|\le 3$ (see Figure 4).  Since $2/15>3/23$, we have $f_1(v)\geq 4-3\cdot 2/15=18/5>83/23$. If  $N(v)\cap A_1\neq \emptyset$, then $|N(v)\cap V_{ 4^+}|\ge 2$ (see Figure 4). So $d(v)\geq 4-2\cdot1/5=18/5>83/23$.

    Let $v\in V_{ 4^+}$ and $d(v)\ge 5$. If  $N(v)\cap (A_1\cup A_3)\neq \emptyset $, then $|N(v)\cap V_{ 4^+}|\ge 2$ (see Figure 4). Thus $f_1(v)\geq d(v)-7/15\cdot(d(v)-2)\geq 18/5$. If $N(v)\cap (A_1\cup A_3)=\emptyset$, then $|N(v)\cap V_{ 4^+}|\ge 1$ (see Figure 4). So $f_1(v)\geq d(v)-8/23(d(v)-1)\geq 83/23.$\q

    From Claim 7, we have that
    $$e(G)=\frac{1}{2}\sum_{v\in V(G)}d(v)=\frac{1}{2}\sum_{v\in V(G)}f_0(v)=\frac{1}{2}\sum_{v\in V(G)}f_1(v)\geq 83n/46.$$
    So we have $e(G)\geq \lc 83n/46\rc \geq \lc 7n/4\rc$ when $n\geq 7$.
Thus the end of the proof of Theorem \ref{thm:stronger result}.
\end{reprooff}
\section{A construction of an edge-pancyclic graph}\label{sec:section 3}
We need three steps to construct an edge-pancyclic graph $G$. In this section, let $[n]=\{1,2,\dots,n\}$ and $[a,b]=\{a,\ldots,b\}$ where $a,b\in [n]$ and $a\le b$.

 \noindent\textbf{Step 1.} Construct the graph $G_1$.

Let $s\geq 2$ be an integer and set $\ell =\lf\frac{s}{\ln s}\rf$. Then  $s^\ell \leq e^s<s^{\ell+1}$. Let $C= v_1v_2\ldots v_{s^\ell}v_1$ be a  cycle. The graph $G_1$ is obtained from $C$ by adding edges $v_{js}v_{js+s^i}$, where $j\in [s^{\ell-1}]$ and $i\in [\ell-1]$. We make a deal that when $s^\ell<k\leq 2s^\ell$, $v_k=v_{k-s^\ell}$, and when $-s^\ell<k\leq 0$, $v_k=v_{k+s^\ell}$. Denote $E_1=E(C)$ and $E_2=E(G_1)\setminus E_1$. Then $G_1$  has the following properties.

\begin{lemma}\label{lem: each edge in a cycle s^p}
    For each edge $e\in E(G_1)$ and each integer $p\in [\ell-1]$, there exists a cycle contained $e$ with length in $[s^p-s^{p-1},s^p+3]$. Furthermore, we can also require at most $3$ edges on the cycle to be in $E_2$.
\end{lemma}
\begin{proof} Let $e\in E(G_1)$ and  $p\in [\ell-1]$.
    First, we consider the case $e\in E_1$. Notice that every edge in $E_1$ lies in the cycle $C_{s^p+1}=v_{js}v_{js+1}\ldots v_{js+s^p}v_{js}$ for some $j\in [s^{\ell-1}]$ and each $p\in [\ell-1]$. Moreover, $|E(C_{s^p+1})\cap E_2|=1$. Thus the lemma holds for every $e\in E_1$.

    Now we consider the case $e\in E_2$. By symmetry, it suffices to consider the edge $e=v_{js}v_{js+s^q}$, where $q\in [\ell-1]$. We will finish the proof by considering the relationship between $p$ and $q$.
    \begin{mycase}
        \case $p=q$.
        In this case, $e$ is contained in a cycle $v_{js}v_{js+1}\ldots v_{js+s^p}v_{js}$ of length $s^p+1$ and one edge comes from $E_2$ (see Figure 5).

        \case $p>q$. Then  $v_{js+s^p}v_{js} v_{is+s^q}v_{js+s^q+1}v_{js+s^q+2}\ldots v_{js+s^p}$ is a cycle contained $e$ of length $s^p-s^q+2$ with two edges in $E_2$ (see Figure 5).

        \case $p<q$. Then  $v_{js+s^q}v_{js}v_{js+1}v_{js+2}\ldots v_{js+s^p}v_{js+s^p+s^q}v_{js+s^q}$ is a cycle contained $e$ of length $s^p+3$ with three edges in $E_2$ (see Figure 5).
    \end{mycase}
    \begin{figure}[h]
        \centering
        \includegraphics[width=0.8\linewidth]{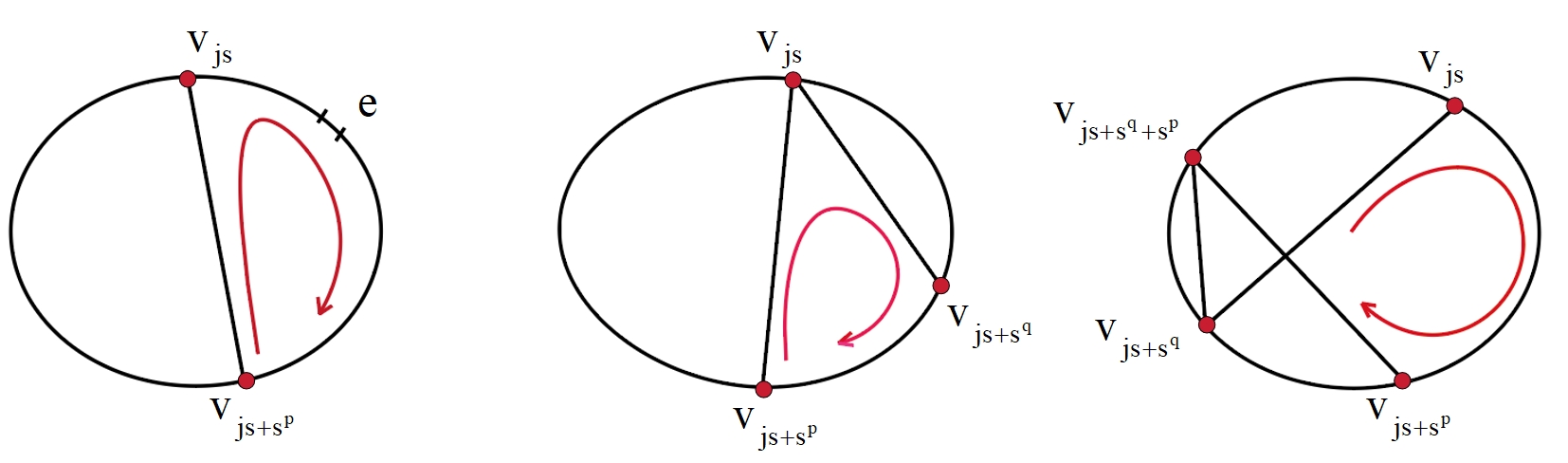}
        \caption{Three types of cycles in Lemma \ref{lem: each edge in a cycle s^p}}
        \label{fig:enter-label}
    \end{figure}
\end{proof}
\textbf{Step 2.} Replace every $e\in E_1$ with a graph.

Let $F_a$ be the fan graph of order $a$, which is obtained from a path $P_{a-1}$ by joining every vertex in $P_{a-1}$ to a new vertex. Let $A$ and $B$ be two disjoint union of $F_{50s}$ with $V(A)=\{w_1,\ldots,w_{50s-1},w\}$ and $V(B)=\{u_1,\ldots,u_{50s-1},u\}$ where $w$ and $u$ are the center vertices of $A$ and $B$ respectively.

Let $H(s)$ be the graph obtained from the union of $A$ and $B$ by adding edges $wu_1,uw_{50s-1}$ and $uw$ (see Figure 6).

\begin{figure}[t]
    \centering
    \includegraphics[width=0.9\linewidth]{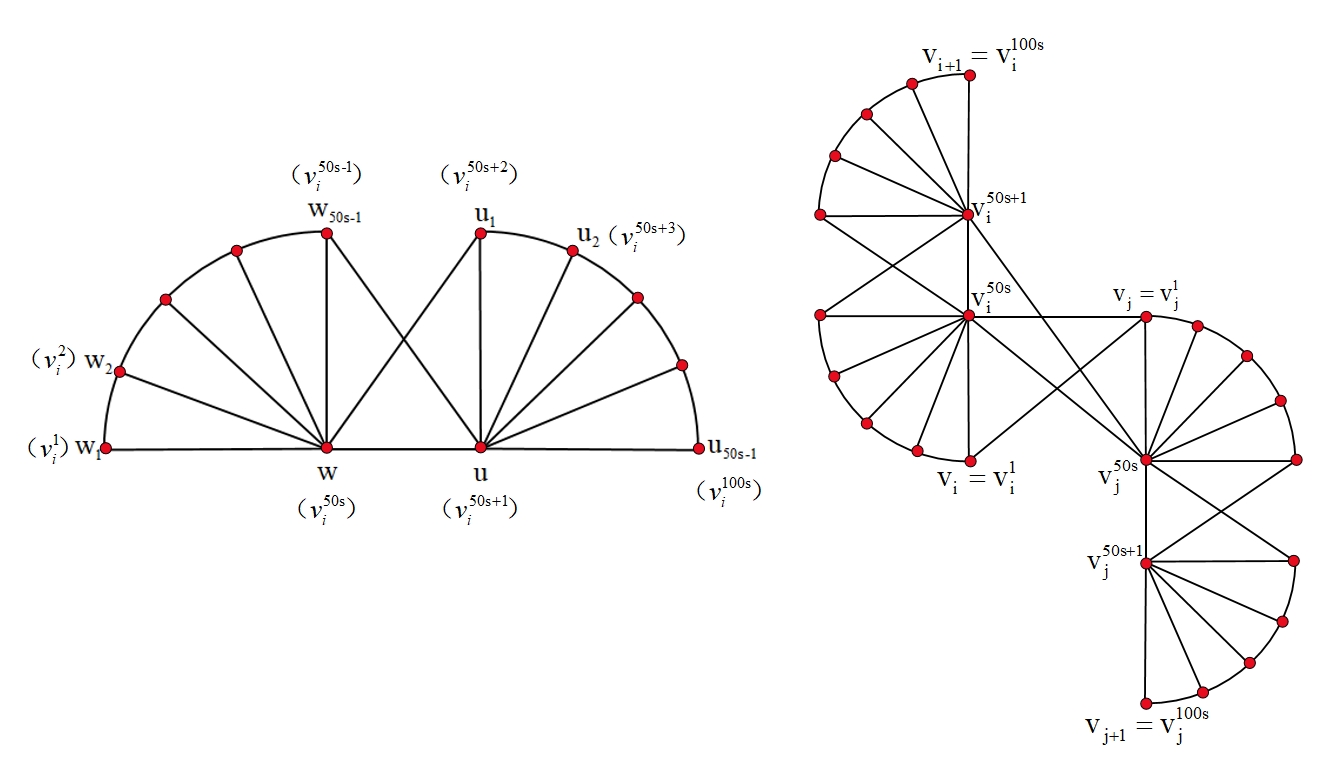}
    \caption{The graph $H(s)$ and the replacement of $G_1$}
    \label{fig:enter-label}
\end{figure}

Then for every edge $v_iv_{i+1}\in E_1$, we replace it by a copy of $H(s)$ with $v_i=w_1$ and $v_{i+1}=u_{50s+1}$. The vertex $w_j$ is renamed as $v_i^j$ for $j\in [50s-1]$, $w$ is renamed as $v_i^{50s}$ and $u$ is renamed as $v_i^{50s+1}$, while $u_t$ is renamed as $v_i^{50s+1+t}$ for $t\in [50s-1]$. Thus, we have $v_i=v_i^1$ and $v_{i+1}=v_i^{100s}$.
The new graph is denoted as $G_2$.

 \noindent\textbf{Step 3.} Replace every $e\in E_2$ with four edges.

For every $e=v_iv_j(=v_i^1v_j^1)\in E_2$ and $v_i$ appears before $v_j$ clockwise (we say $v_i$ appears before $v_j$ clockwise if the number of vertices in $\{v_i,v_{i+1},\dots,v_{j-1},v_j\}$ is at most $s^\ell/2$), we add the edges $v_i^{50s}v_j^1,v_i^{50s}v_j^{50s},v_i^{50s+1}v_j^{50s}$ to $G_2$ (see Figure 6). The new graph is denoted as $G$.

It is easy to calculate that $v(G)=(100s-1)s^\ell$ and $e(G)=2v(G)-|E_1|+4|E_2|$. Since $|E_1|=s^\ell$ and $|E_2|=(\ell-1)s^{\ell-1}\leq \frac{1}{8}|E(C)|$ when $s\geq e^8$, we have
$$2v(G)-e(G)\geq \frac{1}{2}|E(C)|=\frac{n}{2(100s-1)}\geq \frac{n}{200s}.$$
Since $s^\ell\leq e^s<s^{\ell+1}$ and $\ln n\geq \ln(99s\cdot s^\ell)\geq \ln 99+\ln(e^s)\geq s$, we have
$$e(G)\leq 2n-\frac{n}{200\ln n}.$$
To prove Theorem \ref{thm: upper bound}, it is left to prove $G$ is edge-pancyclic.

\begin{theorem}
    The graph $G$ is edge-pancyclic.
\end{theorem}
\begin{proof}
    We denote the edges of $G$ obtained from Step 2 and Step 3 by $E_3$ and  $E_4$, respectively.
Given $1\le i\le s^\ell$, let $\tau_i^1,\tau_i^2$ be two integers that are divisible by $s$, such that $\tau_i^1< i< \tau_i^2$ and $\tau_i^2-\tau_i^1=2s$. We have the following claims.

     \noindent{\bf Claim 8.}
        Every edge in $E_3$ lies in a cycle of length $k$, where $3\leq k\leq 600s$.

    \noindent{\bf Proof of Claim 8.}
        It is easy to check that every edge in $H(s)$ lies in a cycle of length $k$, $3\leq k\leq 50s$.

        For an edge $v_i^jv_i^{j+1}$, where $1\le i\le s^\ell$ and $1\leq j\leq 25s$, we have  $$C^\star :=v_i^1 \ldots v_i^{j-1}v_i^jv_i^{j+1}v_i^{50s}v_i^{50s+1}v_i^{100s}(=v_{i+1}^{1})v_{i+1}^{50s}v_{i+1}^{50s+1}v_{i+1}^{100s}\dots v_{\tau_i^2}^{50s}v_{\tau_i^2+s^{\ell-1}}^{50s}v_{\tau_i^2+s^{\ell-1}-s}^{50s}$$
        $$v_{\tau_i^2+s^{\ell-1}-2s}^{50s}v_{\tau_i^2+2s^{\ell-1}-2s}^{50s}\dots v_{\tau_i^2+s\cdot s^{\ell-1}-2s}^{50s}(=v_{\tau_i^1}^{50s})v_{\tau_i^1}^{50s+1}v_{\tau_i^1}^{100s}(=v_{\tau_i^1+1}^{1})v_{\tau_i^1+1}^{50s}v_{\tau_i^1+1}^{50s+1}v_{\tau_i^1+1}^{100s}\dots v_i^1$$
        is a cycle of length $j+7s+2$ (see Figure 8).

        Note that each $P_4=v_p^{1} v_p^{50s}  v_p^{50s+1} v_p^{100s}$ in $C^\star$ can be replaced by a path of each length between $3$ and $100s-1$ that starts from $v_p^{1}$ and ends with $v_p^{100s}$  (see Figure 8). And there are $2s-3$ copies of such $P_4$ in $C^\star$. Thus, every edge $v_i^jv_i^{j+1}$ for $1\leq j\leq 25s$ is contained in a cycle of length $k$ for $3\leq k\leq 600s$ (In fact, the above proof implies a wider range than $3\leq k\leq 600s$, but $600s$ is enough for this proof).

        When $ 25s\le j\le 50s-1$, we can replace the segment $v_i^1v_i^2\dots v_i^{j+1}v_i^{50s}v_i^{50s+1}$ in $C^\star$ with $v_i^1v_i^{50s}v_i^jv_i^{j+1}\\ \dots v_i^{50s-1}v_i^{50s+1}$ and with a similar replacement to prove that  every edge $v_i^jv_i^{j+1}$ for $25s\leq j\leq 50s$ is contained in a cycle of length $k$ for $3\leq k\leq 600s$. 

        By symmetry and those two cycles, under a similar replacement argument, it is easy to check that we can prove that each edge in $E_3$ can be contained in a cycle of length $k$ for $3\leq k\leq 600s$.

        Thus, every edge in $E_3$ is contained in a cycle of length $k$, $3\leq k\leq 600s$.\q
        \begin{figure}
            \centering
        \includegraphics[width=0.7\linewidth]{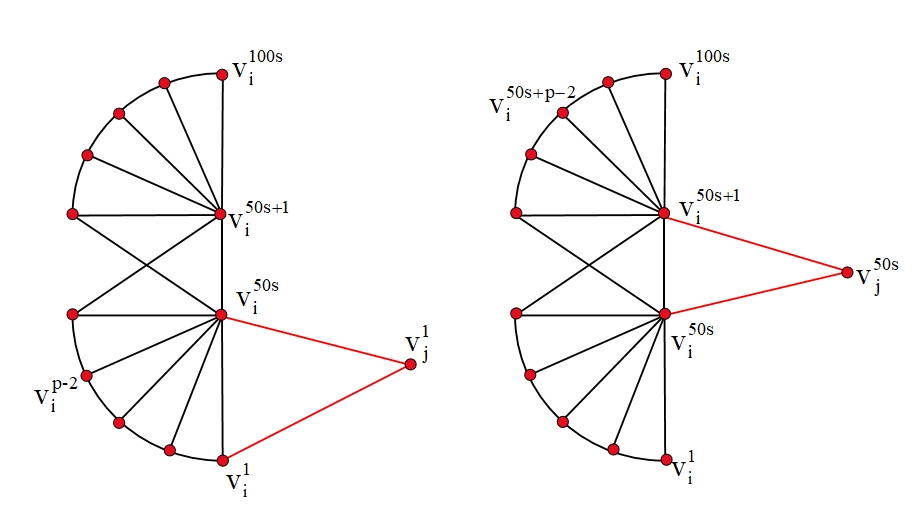}
            \caption{Every edge in $E_4$ lies in a cycle of each length in $[3,10s]$}
            \label{fig:enter-label}
        \end{figure}
\begin{figure}
            \centering
            \includegraphics[width=0.9\linewidth]{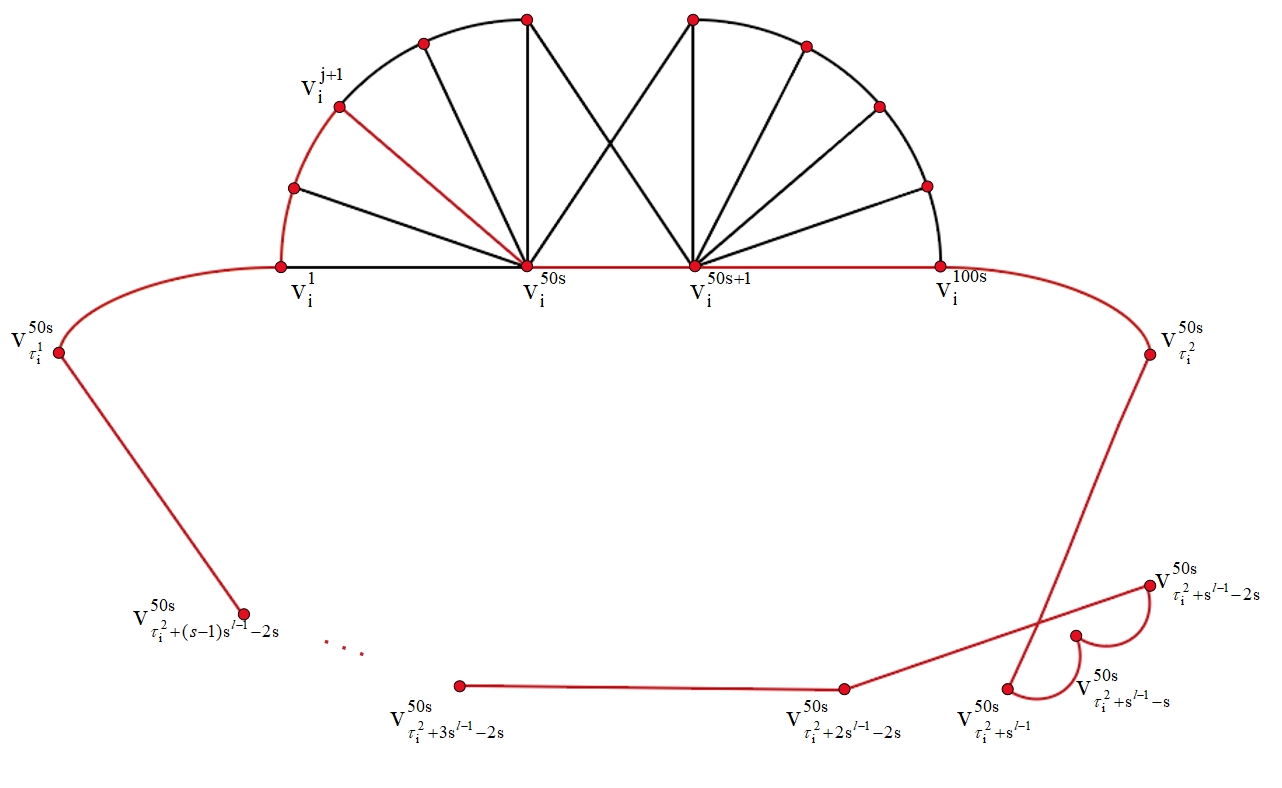}
            \caption{When $1\leq j\le 25s$, the edge $v_i^jv_i^{j+1}$ of $G$ lies in a cycle of each length in $[3,600s]$}
            \label{fig:enter-label}
        \end{figure}

   \noindent{\bf Claim 9.}
        Every edge in $E_4$ lies in a cycle of  length $k$, $3\leq k\leq 600s$.

   \noindent{\bf Proof of Claim 9.} Note that $$E_4=\cup_{i,j\in [s^\ell],v_iv_j\in E_2}\{v_i^{1}v_j^1,v_i^{50s}v_j^{1},v_i^{50s}v_j^{50s},v_i^{50s+1}v_j^{50s}\}.$$
        Consider the edges $v_i^{1}v_j^1,v_i^{50s}v_j^{1}\in E_4$, where $1\le i,j\le s^\ell$. For each $p\in [3,50s]$, $$v_j^{1}v_i^1v_i^2\ldots v_i^{p-2}v_i^{50s}v_j^{1}$$ is a cycle of length $p$ containing $v_i^{1}v_j^1$ and $v_i^{50s}v_j^{1}$ (see Figure 7).

        Consider the edges $v_i^{50s}v_j^{50s},v_i^{50s+1}v_j^{50s}\in E_4$, where $1\le i,j\le s^\ell$. For each $4\leq p\leq 50s$, $$v_j^{50s}v_i^{50s}v_i^{50s+2}\ldots v_i^{50s+p-2}v_i^{50s+1}v_j^{50s}$$ is a cycle  of length $p$ containing $v_i^{50s}v_j^{50s}$ and $v_i^{50s+1}v_j^{50s}$  (see Figure 7). And $v_j^{50s}v_i^{50s}v_i^{50s+1}v_j^{50s}$ is a $C_3$ containing both edges  (see Figure 7).

        Then we can use the similar cycle $C^\star$ as in Claim 8 and the similar replacement to prove every edge in $E_4$ is contained in a cycle of length $k$ if $j\not\in \{\tau_i^2,\tau_i^2+s^{\ell-1},\tau_i^2+s^{\ell-1}-s,\tau_i^2+s^{\ell-1}-2s,\tau_i^2+2s^{\ell-1}-2s,\dots,\tau^1_i\}$, where $3\leq k\leq 600s$.
        When $j\in \{\tau_i^2,\tau_i^2+s^{\ell-1},\tau_i^2+s^{\ell-1}-s,\tau_i^2+s^{\ell-1}-2s,\tau_i^2+2s^{\ell-1}-2s,\dots,v_{\tau^1_i}\}$, we can replace $\tau^2_i$ with $\hat{\tau}_i^2 :=\tau^2_i+s$ and add a vertex $v_{\hat{\tau}_i^2+s^{\ell-1}-3s}$ between $v_{\hat{\tau}_i^2+s^{\ell-1}-2s}$ and $v_{\hat{\tau}_i^2+2s^{\ell-1}-2s}$ in $C^\star$. Then we use a similar replacement to prove such edges in $E_4$ are contained in a cycle of length $k$, $3\leq k\leq 600s$.
        \q

    \noindent{\bf Claim 10.}
        Every edge in $E_3$ lies in a cycle of  length $k$, $ 3s^\ell+100s\leq k\leq v(G)$.

    \noindent{\bf Proof of Claim 10.}
    Notice that for every $3\leq p\leq 100s-1$, there is a path of length $p$ contained in the copy of $H(s)$ in $G$ connecting $v_{i}$ and $v_{i+1}$, where $i\in [s^\ell]$ and $v_{i},v_{i+1}\in V(G_1)$. Moreover, every edge in $E_3$ lies in a path of length $100s-1$  contained in the copy of $H(s)$ in $G$ connecting $v_j$ and $v_{j+1}$ for some $j\in [s^\ell]$, where $v_j,v_{j+1}\in V(G_1)$.
    Thus using the Hamilton cycle in $G_1$, we can find a cycle of each length $k$, $k\geq 3s^\ell+100s$ containing $e$ for every $e\in E_3$.\q

    \noindent{\bf Claim 11.}
        Every edge in $E_4$ lies in a cycle of each length $k$, $ 3s^\ell+200s\leq k\leq v(G)$.

    \noindent{\bf Proof of Claim 11.}
        Notice that for the edges $v_i^{50}v_j^1, v_i^{50s+1}v_j^{50s}\in E_4$,
        \begin{equation*}
            \begin{aligned}
            &v_i^{50s+1}v_i^{50s+2}\dots v_i^{100s}(=v_{i+1}^1)v_{i+1}^{50s}v_{i+1}^{50s+1}v_{i+1}^{100s}\dots v_j^1v_i^{50s}v_i^{50s-1}\dots v_i^1v_{i-1}^{50s+1}v_{i-1}^{50s}v_{i-1}^1\dots\\ &v_{j+1}^{1}(=v_{j}^{100s})v_{j}^{100s-1}\dots v_j^{50s+1}v_j^{50s-1}\dots v_j^{2}v_j^{50s}v_i^{50s+1}
            \end{aligned}
        \end{equation*}
        is a cycle of length $3s^\ell+200s-8$. Moreover, every $P_4=v_p^1v_p^{50s}v_p^{50s+1}v_p^{100s}$ can be replaced with $t $ edges, $3\leq t\leq 100s-1$. And there are $s^{\ell}-2$ copies of such $P_4$ in the cycle.
        Thus $v_i^{50}v_j^1$ and $v_i^{50s+1}v_j^{50s}$ lies in a cycle of each length $k$, $k\geq 3s^\ell +200s.$
        We can similarly find the cycles containing $v_i^1v_j^1,v_i^{50s}v_j^{50s}\in E_4$  (see Figure 9).
        \begin{figure}
            \centering
            \includegraphics[width=1\linewidth]{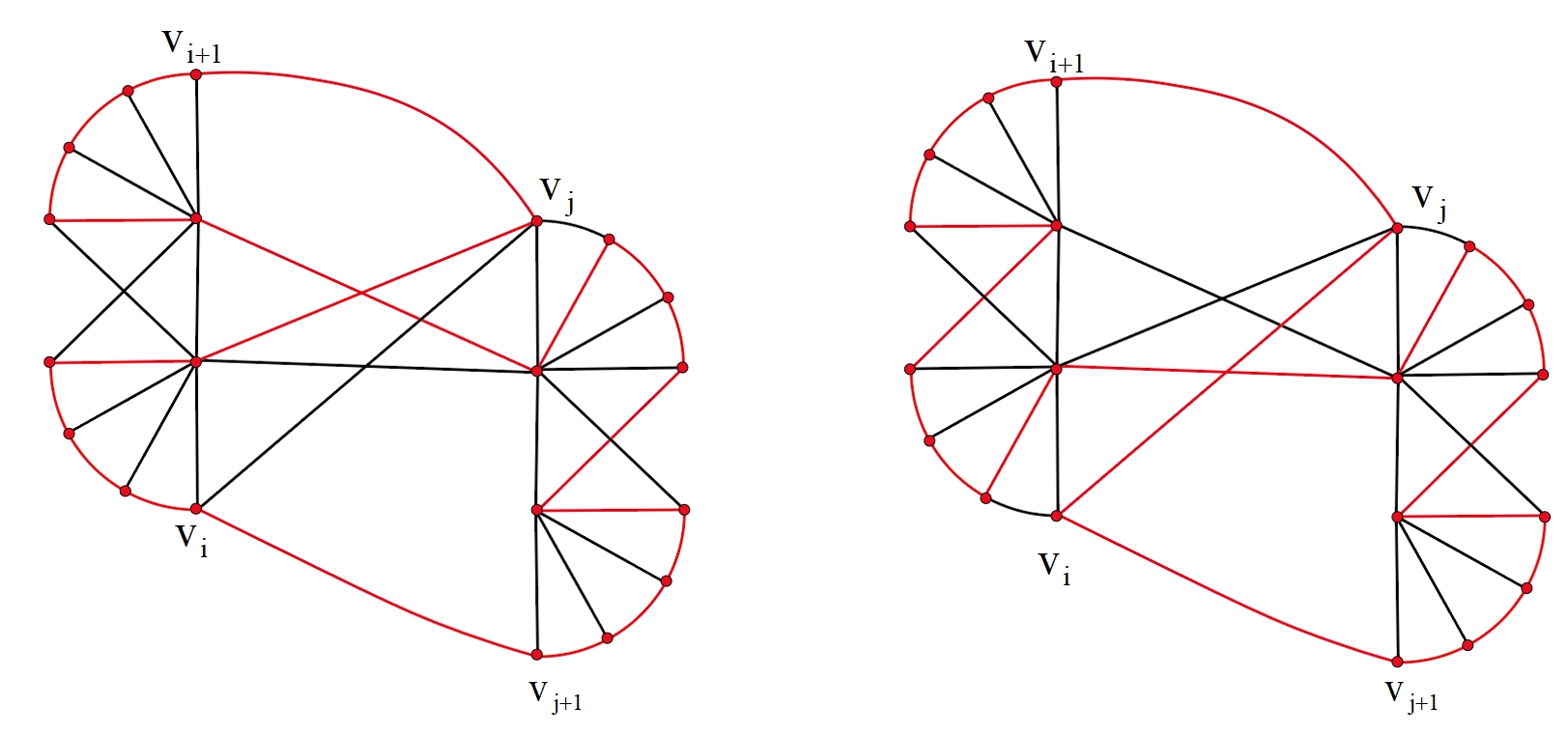}
            \caption{Every edge in $E_4$ lies in a cycle of each length in $[3s^\ell+200s,v(G)]$}
            \label{fig: every E_4}
        \end{figure}\q

    \noindent{\bf Claim 12.}
        Each edge of $G$ lies in a cycle of each length in $[3s^p+500s-6,100s^{p+1}+100s^p-500s+14]$ for each $p\in [\ell-1]$.

    \noindent{\bf Proof of Claim 12.}
        For each $p\in [\ell-1]$, by Lemma \ref{lem: each edge in a cycle s^p}, each edge in $G_1$ is contained in a cycle $C_0$ of some length in $[s^p-s^{p-1},s^p+3]$. Also, at least $s^p-s^{p-1}-3$ edges come from $E_1$.

        Then we can lift the cycle $C_0$ to the graph $G$. It is  clear that every edge in $E_1$ is replaced by a path with length of at least $3$ and at most $100s$ in $G$. There are many ways to do this replacement. Except for at most $5$ edges in the $C_0$ (at most $3$ edges in $E_2$ and at most $2$ edges incident to the given edge), we can replace each of the other edges with a path having at least $3$ at most $100s$ edges. Then we have cycles of each length in $[(s^p+3-5)\cdot3+5\cdot 100s,(s^p-s^{p-1}-5)\cdot 100s+5\cdot 3]$.\q

     When $s\geq e^8$, we have $3s^{p+1}+500s-6\leq 100s^{p+1}+100s^p-500s+14$ for $p\in[\ell-1]$. Thus by Claim 12, each edge of $G$ lies in a cycle of  length $k$ for any $k\in [503s-6,100s^\ell+100s^{\ell-1}-500s+14]$.
     Note that $3s^\ell+200s\leq 100s^{\ell}+100s^{\ell-1}-500s+14$. By Claims 10 and  11, each edge of $G$ lies in a cycle of  length $k$ for any $k\in [503s-6,v(G)]$. Together with Claims 8 and 9, $G$ is edge-pancyclic.
\end{proof}

\section*{Acknowledgement}
Many thanks to the anonymous referees for their many helpful comments and suggestions, which have considerably improved the presentation of the paper.
This work is supported by the National Natural Science Foundation of China~(Grant 12171272 \& 12426603).

\end{document}